\documentclass[a4paper,12pt]{amsart}
\usepackage[T1]{fontenc}    
\usepackage{enumerate}
\usepackage{amssymb}
\usepackage{tikz, subfigure}
\usetikzlibrary{patterns}

\def\eps{\varepsilon }
\def\lam{\lambda }

\def\uu{\mathcal U}

\newcommand{\set}[1]{\left\lbrace #1\right\rbrace}
\providecommand{\abs}[1]{\left\lvert#1\right\rvert}

\newcommand{\remove}[1]{ }
\newcommand{\qtq}[1]{\quad\text{#1}\quad}

\newtheorem{theorem}{Theorem}[section]
\newtheorem{lemma}[theorem]{Lemma}

\newtheorem*{2}{Theorem 1.1}

\theoremstyle{definition}

\theoremstyle{remark}
\newtheorem*{remark}{Remark}
\remove{}

\numberwithin{equation}{section}

\begin{document}
\title[]{Intersections of Siepinski  gasket with its translation }
\author{Yi Cai*
, Wenxia Li}
\address{School of Mathematical Sciences, East China Normal University, Shanghai 200241, People's Republic of China}
\email{52170601013@stu.ecnu.edu.cn}
\address{School of Mathematical Sciences, Shanghai Key Laboratory of PMMP, East China Normal University, Shanghai 200062,
People's Republic of China}
\email{wxli@math.ecnu.edu.cn}
\date{Version 2019-07-31}
\thanks{*Corresponding author}

\begin{abstract}
Let $E$ be  the Sierpinski  gasket, i.e., the self-similar set generated by the IFS $\left \{f_a(x)=\frac{x+a}{q}: a\in \{(0,0), (0,1), (1,0)\}\right \}$.  In paper, we provide a description of the following set for $2<q<3$
\begin{equation*}
D_q=\{\dim _H(E\cap (E+t)):\;t\in T\},
\end{equation*}
where $T$ is the set of $t=(t_1, t_2)$ with $t\in E-E$ and $t_1, t_2$ have unique $q$-expansions w.r.t $\set{-1,0,1}$.
\end{abstract}
\maketitle
\section{Introduction}

The Siepinski gasket in $\mathbb R^2$ is defined as follows:
\begin{equation}\label{1}
E=\set{\sum_{i=1}^{\infty}\frac{c_i}{q^i}:c_i\in\Omega _1}
\end{equation}
where $\Omega _1:=\set{(0,0),(0,1),(1,0)}$, which is a self-similar set generated by the \emph{iterated function system}
(IFS)
\begin{equation*}
f_a(x)=\frac{x+a}{q}\;\;\textrm{with}\;a\in \Omega _1, x\in \mathbb R^2,
\end{equation*}
i.e., $E=\bigcup _{a\in \Omega_1 }f_a(E)$.

In the past twenty years many works have been devoted to the Siepinski gasket; e.g., see \cite{BMS,DL,NS}. Sidorov \cite{NS} discovered for $q\in(1,2)$ the IFS $\set{f_a:a\in \Omega _1}$ fails the open set condition. Furthermore, for $q\in(1,3/2]$ the Sierpinski gasket coincides with its convex hull. However, for $q\in(3/2,2]$ descriping the sturcture of this set becomes difficult. When $q=2$, the IFS satisfies the open set condition, but does not satisfy the strong separation condition. When $q>2$, the IFS satisfies the strong separation condition (cf. \cite{KF}).

In this paper we investigate the intersection of $E$ with its translation under the assumption  $2<q<3$. In this case, the IFS $\{f_a: a\in \Omega _1\}$ satisfies the strong separation condition hence the Hausdorff dimension of $E$ $\dim _HE=\log 3/\log q$ (cf. \cite{KF}).

One can check that
\begin{equation*}
E\cap (E+t)\ne \emptyset \;\;\textrm{if and only if}\;\; t\in E-E,
\end{equation*}
where the difference set $E-E=\{x-y: x, y\in E\}$. According to \eqref{1} we can rewrite the difference set as
\begin{equation*}
\begin{split}
E-E&=\left \{\sum _{i=1}^\infty \frac{x_i}{q^i}-\sum _{i=1}^\infty \frac{y_i}{q^i}: x_i, y_i\in \Omega _1\right \}\\
\; &=\left \{\sum _{i=1}^\infty \frac{t_i}{q^i}: t_i\in \Omega _2 \right \}
\end{split}
\end{equation*}
where
\begin{align*}
\Omega _2&=\Omega _1-\Omega _1\\
&=\set{(0,0),(0,1),(1,0),(-1,0),(-1,1),(0,-1),(1,-1)}.
\end{align*}
Hence, for any $t\in E-E$, there exist at least
one $(t_i)\in \Omega _2^{\mathbb N}$ such that
$$
t=\sum _{i=1}^\infty \frac{t_i}{q^i}.
$$
The above sequence $(t_i)$ is called a \emph{$q$-expansion} of $t$  with respect to the digit set $\Omega _2$ in base $q$.
 Note that $t$ may have multiple expansions.
On the other hand, for any $t\in E-E$, we have (cf. \cite{WLDX})
\begin{equation}\label{e11}
E\cap(E+t)=\bigcup _{(t_i)}\set{\sum_{i=1}^{\infty}\frac{a_i}{q^i}:a_i\in(\Omega _1\cap(\Omega _1+t_i))},
\end{equation}
where $(t_i)$ takes over all possible $q$-expansions of $t$ with respect to (the digit set) $\Omega _2$.

For each $t_i\in \Omega _2$, we denote $t_i=(t_{i,1}, t_{i,2})$ where  $t_{i,1}, t_{i,2}\in \{-1, 0, 1\}$.
Thus for $t\in E-E$ with an expansion $(t_i)\in \Omega _2^{\mathbb N}$ one has
\begin{equation*}
t=\sum _{i=1}^\infty \frac{t_i}{q^i}=\left (\sum _{i=1}^\infty \frac{t_{i,1}}{q^i},
\sum _{i=1}^\infty \frac{t_{i,2}}{q^i}\right ).
\end{equation*}
So we know that if $(t_i)$ is an expansion of $t=(t_1, t_2)$ with respect to  $\Omega _2$, then
$(t_{i,1}), (t_{i,2})$ are the expansions of $t_1$ and $t_2$ with respect to  $\Omega_3:=\{-1,0,1\}$. We say
two sequences $(t_{i,1}), (t_{i,2})\in \Omega_3^\mathbb N$ are \emph{$\Omega_2$-matched} if $(t_{i,1}, t_{i,2})\in \Omega _2^{\mathbb N}$
for all $i\ge1$.

Note that $2<q<3$, which implies that
\begin{equation*}
\left \{\sum _{i=1}^\infty \frac{s_i}{q^i}: s_i\in \Omega_3\right \}=
\left [-\frac{1}{q-1}, \frac{1}{q-1}\right ].
\end{equation*}
 Thus each $s\in I_q:=\left [-1/(q-1), 1/(q-1)\right ]$ has at least
one $q$-expansion with respect to  $\Omega_3$. This kind of expansions in  \emph{non-integer} base $q$
was first studied by   R\'enyi \cite{R}.
Since then many results on  $q$-expansion  have been discovered. We denote by $\mathcal U_q$ the set
of $s\in I_q$ which has a unique $q$-expansion with respect to  $\Omega_3$. Now let $t=(t_1, t_2)$ and
\begin{equation*}
T=\{t\in E-E: t_1, t_2\in {\mathcal U}_q\}.
\end{equation*}
In this paper, we try to determine the set
\begin{equation*}
D_q:=\set{\dim _H(E\cap (E+t)):\;t\in T}.
\end{equation*}
This is motivated by a recent work
of Baker and Kong
\cite{BD}. They mainly studied the set $\{\dim _H(F\cap (F+s)):\;s\in \mathcal U_q\}$,
where $F=\{\sum _{i=1}^\infty \frac{x_i}{q^i}:
x_i\in \{0,1\}\}$.

Before we state our result, first let us introduce some notation. Let $A:=\set{a_1,a_2,a_3}$ be a set of consecutive integers with $a_1<a_2<a_3$, for a word $w=b_1\cdots b_n\in A^n$ we write
\begin{align*}
&w^{-}:=b_1\cdots b_{n-1}(b_n-1)\qtq{if}b_n>a_1,\\
&w^{+}:=b_1\cdots b_{n-1}(b_n+1)\qtq{if}b_n<a_3,\\
&w^\infty:=www\cdots,
\end{align*}
and $\abs w$ denotes the length of $w$.
Next we recall some results which are useful in this paper.
Komornik et al. \cite{VKL,KLP-2011,K-2012,KL,KL2} discovered
that there exist a smallest base $q_{KL}$ (called the \emph{Komornik--Loreti constant}) in which $1$ has a unique $q$-expansion
with respect to $\{0,1,2\}$,
 and a bases sequence $(q_n)_{n=1}^\infty$. Actually $q_n$ is the base which $1$ has a greedy expansion of the form $w_n0^\infty$. We define recursively a sequence of words, starting with $w_1=2$
\begin{equation*}
w_{n+1}:=(w_n\overline{w_n})^+,\quad n=1,2,\ldots
\end{equation*}
 where the reflection of $w_n$ with respect to $\set{0,1,2}$ is defined as $\overline{w_n}=\overline{b_1\cdots b_n}=(2-b_1)\cdots (2-b_n)$. Then from the definition above we know that $(q_n)$ is strictly increasing and has the limit $q_{KL}$, for details we refer the reader to \cite{VKL}.

Now we state our main result.
\begin{theorem}\label{2}\mbox{}
\begin{enumerate}[\upshape (i)]
\item If $q\in(q_m,q_{m+1}]$ for some $m\ge1$, then
\begin{equation*}
D_q=\set{0,\frac{\log3}{\log q}}\bigcup\set{-\frac{\log3}{\log q}\sum_{i=1}^{n}\left(-\frac{1}{2}\right)^i:1\le n<m}.
\end{equation*}
\item  If $q=q_{KL}$ then
\begin{align*}
 D_{q_{KL}}&=\set{0,\frac{\log3}{\log q_{KL}},\frac{\log3}{3\log q_{KL}}}\\
 &\bigcup\set{-\frac{\log3}{\log q_{KL}}\sum_{i=1}^{n}
 \left(-\frac{1}{2}\right)^i:1\le n<\infty}.
\end{align*}
\item  If $q\in(q_{KL},3)$, $D_q$ contains an interval.
\end{enumerate}
\end{theorem}

The remaining part of this paper is arranged as follows. In next section, we introduce some results needed for our proof. The proof of Theorem
\ref{2} is given in section 3.

\section{Review of some preliminary results}
We first recall the classical Thue--Morse sequence $(\tau_i)_{i=0}^{\infty}$, $\tau_0=0$, $\tau_{2i}=\tau_i$ and $\tau_{2i+1}=1-\tau_i$. This sequence begins with
\begin{equation*}
0110\;1001\;1001\;0110\cdots.
\end{equation*}
Now we introduce the sequence $(\lam_i)_{i=1}^{\infty}\in\Omega_3^\mathbb N$ (see \cite{BD,KL2})
\begin{equation}\label{3}
\lam_i=\tau_i-\tau_{i-1},
\end{equation}
which starts with
\begin{equation*}
10(-1)1\;(-1)010\;(-1)01(-1)\;10(-1)1\cdots.
\end{equation*}
 By \eqref{3} we have
 the following property
\begin{align}\label{13}
\lam_1=1,\quad\lam_{2^{n+1}}=1-\lam_{2^n},\;
\lam_{2^{n}+i}=-\lam_i\qtq{for all $1\le i<2^n$.}
\end{align}
For $n=0,1,2,\cdots $, let
\begin{equation}\label{epsn}
\eps_n=\lam_1\cdots\lam_{2^n}.
\end{equation}
Associating \eqref{13} with \eqref{epsn}, the following simple properties can be verified directly:

(P1) $\eps_0=1$ and $\eps_{n+1}=\eps_n \overline{\eps_n}^+$,
where $\overline{\eps_n}=(-\lam_1)\cdots(-\lam_{2^n})$ is defined as the reflection of $\eps_n$ with respect to $\Omega_3$;

 (P2)  $\eps_{2n+1}$ ends with $0$, $\eps_{2n}$ ends with $1$;

 (P3) $\eps_n$ is the block of length $2^n$. $\eps_n$ contains $0$ for all $n\geq 1$;

 (P4) $\eps_n$ begins with $1$. The odd term of $\eps_n$ is nonzero. There exists $0<2k<2^n$
 such that the $2k$-th term of
 $\eps_n$ is nonzero when $n\geq 3$.

Recall that
\begin{equation*}
\uu_q=\left \{s\in I_q: \textrm{$s$ has unique expansion
w.r.t }\;\Omega_3 \right \}.
\end{equation*}
Let $\uu_q'$ be the set of  corresponding $q$-expansions of elements in $\uu_q$, i.e.,
\begin{equation*}
\uu_q':=\set{(t_i)\in\Omega_3^{\mathbb N}: \sum_{i=1}^{\infty}\frac{t_i}{q^i}\in {\uu_q}}.
\end{equation*}
For two sequences $(t_i)\in \Omega _2^{\mathbb N}$  and $(t_i')\in \Omega_3^{\mathbb N}$, let
\begin{align}\label{tidence}
&d((t_i)):=\liminf_{n\to\infty}\frac{\#\set{1\le i\le n: t_i=(0,0)}}{n},\\
&{d}^*((t_i')):=\liminf_{n\to\infty}\frac{\#\set{1\le i\le n: t_i=0}}{n}\notag.
\end{align}
For two words $(t_i)_{i=1}^n\in \Omega _2^{n}$ and $(t_i')_{i=1}^n\in \Omega_3^{n}$, let
\begin{align}\label{tidencef1}
&{d}((t_i)_{i=1}^n):=\frac{\#\set{1\le i\le n: t_i=(0,0)}}{n},\notag\\
&{d}^*((t_i')_{i=1}^n):=\frac{\#\set{1\le i\le n: t_i=0}}{n}.
\end{align}
The following result can be deduced from \cite[Theorem 3.1]{WLDX} (see \cite{H,HL} for more details) and \eqref{e11} .
\begin{lemma}\label{17}
Let $q\in(2,3)$. Then for any $t\in T$ with $q$-expansion $(t_i)\in \Omega _2^{\mathbb N}$
\begin{equation*}
\dim_H(E\cap(E+t))=\frac{\log3}{\log q}{d}((t_i)),
\end{equation*}
where $d((t_i))$ is given by (\ref{tidence}).
\end{lemma}
The technical lemmas given below are necessary.
\begin{lemma}\label{7}  \cite[Lemma 3.1]{BD}
Let $\eps_n$ be given by (\ref{epsn}). Then
\begin{equation*}
d^*((\eps_n))=-\sum_{i=1}^{n}\left(-\frac{1}{2}\right)^i
\end{equation*}
for all nonnegative integer $n$, where $d^*((\eps_n))$ is given by (\ref{tidencef1}).
\end{lemma}
The authors \cite{KK,DLD} gave a crucial characterisation of the set $\uu_q'$.
\begin{lemma}\label{9}
Let $q\in(q_m,q_{m+1}]$ with $m\ge 1$. Then each element of $\uu_q'\setminus\set{(-1)^\infty,1^\infty}$ ends with one of
\begin{equation*}
0^\infty, (\eps_0\overline{\eps_0})^\infty,(\eps_1\overline{\eps_1})^\infty,(\eps_2\overline{\eps_2})^\infty,\cdots,(\eps_{m-1}\overline{\eps_{m-1}})^\infty,
\end{equation*}
and let $\eps_{-1}\overline{\eps_{-1}}:=00$, for each integer $0\le n\le m$ there exists an element of $\uu_q'\setminus\set{(-1)^\infty,1^\infty}$ ends with $(\eps_{n-1}\overline{\eps_{n-1}})^\infty$.
\end{lemma}

\begin{lemma}\label{10}
Each element of $\uu_{q_{KL}}'\setminus\set{(-1)^\infty,1^\infty}$ is either eventually periodic with one of following period
\begin{equation*}
0^\infty, (\eps_0\overline{\eps_0})^\infty,(\eps_1\overline{\eps_1})^\infty,(\eps_2\overline{\eps_2})^\infty,\cdots,
\end{equation*}
or ends with the sequence of the form
\begin{equation*}
(\eps_{0}\overline{\eps_{0}})^{j_0}(\eps_{0}\overline{\eps_{1}})^{l_0}(\eps_{1}\overline{\eps_{1}})^{j_1}(\eps_{1}\overline{\eps_{2}})^{l_1}\cdots
\end{equation*}
and its reflection, where
\begin{equation*}
l_m\in\set{0,1}\quad\text{and}\quad 0\le j_m<\infty \quad\text{for all $m\ge0$}.
\end{equation*}
Moreover, for each nonnegative integer $n$ there exists an element of $\uu_q'\setminus\set{(-1)^\infty,1^\infty}$ ends with $(\eps_{n-1}\overline{\eps_{n-1}})^\infty$, and for each pair of sequences $(l_i)_{i\ge0},(j_i)_{i\ge0}$ where $l_i\in\set{0,1}$ and $0\le j_i<\infty$, there exists an element of $\uu_q'\setminus\set{(-1)^\infty,1^\infty}$ ends with
\begin{equation*}
(\eps_{0}\overline{\eps_{0}})^{j_0}(\eps_{0}\overline{\eps_{1}})^{l_0}(\eps_{1}\overline{\eps_{1}})^{j_1}(\eps_{1}\overline{\eps_{2}})^{l_1}\cdots.
\end{equation*}
\end{lemma}

\section{Proof of Theorem \ref{2}}
It follows from Lemma \ref{17} that the key to the proof of Theorem \ref{2} is to describe the set $T$ in order to obtain $D_q$.
For sequences $(x_i), (y_i)\in \Omega_3^{\mathbb N}$ and $((x_i, y_i))_{i=1}^\infty\in\Omega_2^\mathbb N$, we use $((x_i), (y_i))$ to denote the sequence
$((x_i, y_i))_{i=1}^\infty $. The following lemmas are helpful for characterizing  $T$.
\begin{lemma}\label{12}
Let $n\ge1$ and let $\sigma$ be the left shift in $\Omega_3^{\mathbb N}$.
\begin{enumerate}[\upshape (i)]

\item If $i=2^n$, then
$$
E_{n,m}^i:=(\sigma^i((\eps_n\overline{\eps_n})^\infty),(\eps_n\overline{\eps_n})^\infty)\in \Omega _2^{\mathbb N}
$$
and this sequence contains infinitely many $(0,0)$.

\item Given arbitrarily $0<i<2^{n+1}$ and $i\neq2^n$.  Then  $E_{n,m}^i$ does not contain $(0,0)$ when $i$ is odd.
And when $i$ is even $E_{n,m}^i\notin \Omega _2^{\mathbb N}$.
\end{enumerate}
\end{lemma}
\begin{proof}
(i) Note that
\begin{align*}
E_{n,n}^{2^n}&=(\sigma^{2^n}((\eps_n\overline{\eps_n})^\infty),(\eps_n\overline{\eps_n})^\infty)\\
&=((\overline{\eps_n}\eps_n)^\infty),(\eps_n\overline{\eps_n})^\infty)\in\Omega_2^{\mathbb N}.
\end{align*}

By (P3) we know that   $E_{n,n}^{2^n}$ contains infinitely many $(0,0)$.

 (ii) We infer from (P4) that $E_{n,n}^i$ does not contain $(0,0)$ when $i$ is odd. So suppose that $i$ is even. Then
 the odd terms in $\sigma^i((\eps_n\overline{\eps_n})^\infty)$ pair the odd terms in $(\eps_n\overline{\eps_n})^\infty$.
 We focus on the odd terms in $(\eps_n\overline{\eps_n})^\infty$. By doing it, let us remove the even terms in
 $(\eps_n\overline{\eps_n})^\infty$.  Let
\begin{equation*}
A_n:=(a_i)_1^{2^n},\; a_i=\lambda_{2i-1}.
\end{equation*}
By \eqref{13} we have $a_1=1, A_1=1(-1), A_{k+1}=a_1\cdots a_{2^{k+1}}=A_k\overline{A_k}$, where $\overline{A_k}=(-a_1)\cdots (-a_{2^k})$. Let
\begin{equation*}
D_{k,k}^i:=(\sigma^i((A_k\overline{A_k})^\infty),(A_k\overline{A_k})^\infty).
\end{equation*}
We claim that $D_{k,k}^i$ contains $(-1,-1)$ or $(1,1)$ for all $0<i<2^{k+1}$ and $i\neq2^k$. We prove it by induction. First we observe that
\begin{align}\label{19}
\begin{split}
&(A_1\overline{A_1})^\infty=(1(-1)(-1)1)^\infty,\\
&(A_2\overline{A_2})^\infty=(1(-1)(-1)1(-1)11(-1))^\infty,\\
\cdots,\\
&(A_{k+1}\overline{A_{k+1}})^\infty=(A_k\overline{A_k}\overline{A_k}A_k)^\infty,\\
\cdots.
\end{split}
\end{align}
When $k=1$, $D_{1,1}^1$ and $D_{1,1}^3$ contain $(-1,-1)$. When $k=2$, $D_{2,2}^i$ contains $(-1,-1)$
 for all $i\neq4$, which can be verified directly. We assume that $D_{k,k}^i$ contains
  $(-1,-1)$ or $(1,1)$ for all $i\neq2^k$ and some $k=r\ge3$. Then we have to show $D_{r+1, r+1}^i$ contains
   $(-1,-1)$ or $(1,1)$ for all $i\neq2^{r+1}$. Applying \eqref{19} for each $i_1\neq2^r$,
    we see the block of length $(2^r,2^r)$ of $D_{r,r}^{i_1}$ in which  $(-1,-1)$ or $(1,1)$
     locates (see Figure 1) will be rearranged in $D_{r+1, r+1}^{i_2}$
     (we denote by $(\lvert a\rvert,\lvert b\rvert)$ the length of block $A=(a,b)\in\Omega_2^n$),
      we split into four cases (see Figure 2). For $k=r+1$, the cases $i=2^r$ and $2^{r+1}+2^r$,
       one can check directly by Figure 2. Hence $D_{k,k}^i$ contains $(-1,-1)$ or $(1,1)$ for all $0<i<2^{k+1}$
       and $i\neq2^k$. We complete the proof by the fact that there exist two odd integers $u,v$ such that
       $(\lambda_u,\lambda_v)=(-1,-1)$ or $(1,1)$, and $E_{n,n}^i$ contains $(\lambda_u,\lambda_v)$ for all
        $0<i<2^{n+1}$ and $i\neq2^n$.
\begin{figure}[h]\label{figure1}
\centering
\begin{tikzpicture}[scale=5]
\draw[pattern=north west lines] (0,.62) rectangle (.2,.7);
\node [label={[xshift=-1.2cm, yshift=2.8cm]\scriptsize $\overline{A_k}$ denotes by}] {};
\draw(.8,.62) rectangle (1,.7);
\node [label={[xshift=2.5cm, yshift=2.8cm]\scriptsize and $A_k$ denotes by}] {};
\draw (-.4,.3) rectangle (-.2,.38);
\draw [pattern=north west lines] (-.2,.3) rectangle (0,.38);
\draw (0,.3) rectangle (.2,.38);
\draw [pattern=north west lines] (.2,.3) rectangle (.4,.38);
\draw (-.32,.2) rectangle (-.12,.28);
\draw [pattern=north west lines] (-.12,.2) rectangle (.08,.28);
\draw [densely dashed](-.32,.28)--(-.32,.48);
\draw [densely dashed](-.12,.28)--(-.12,.48);
\draw [densely dashed](.08,.28)--(.08,.48);
\draw (-.32,.45)--(-.26,.45);
\draw (-.18,.45)--(-.12,.45);
\draw(-.22,.45)node{\scriptsize A};
\draw (-.12,.45)--(-.06,.45);
\draw (.02,.45)--(.08,.45);
\draw(-.02,.45)node{\scriptsize B};
\draw (-.4,-.04) rectangle (-.2,.04);
\draw [pattern=north west lines] (-.2,-.04) rectangle (0,.04);
\draw (0,-.04) rectangle (.2,.04);
\draw [pattern=north west lines] (.2,-.04) rectangle (.4,.04);

\draw (-.12,-.06) rectangle (.08,-.14);
\draw [pattern=north west lines] (.08,-.06) rectangle (.28,-.14);

\draw [densely dashed](-.12,-.14)--(-.12,.12);
\draw [densely dashed](.08,-.14)--(.08,.12);
\draw [densely dashed](.28,-.14)--(.28,.12);
\draw (-.12,.1)--(-.06,.1);
\draw (.02,.1)--(.08,.1);
\draw(-.02,.1)node{\scriptsize C};
\draw (.08,.1)--(.14,.1);
\draw (.22,.1)--(.28,.1);
\draw(.18,.1)node{\scriptsize D};
\node [label={[xshift=-3.6cm, yshift=1.3cm]\scriptsize I}] {};
\node [label={[xshift=-3.6cm, yshift=-.4cm]\scriptsize II}] {};
\end{tikzpicture}\caption{ I: $0<i_1<2^r$ and II: $2^r<i_1<2^{r+1}$.}
\end{figure}
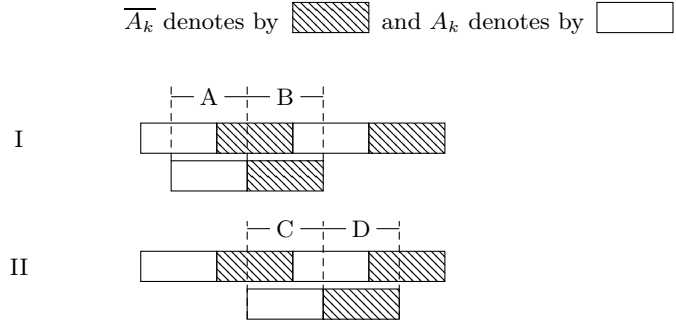

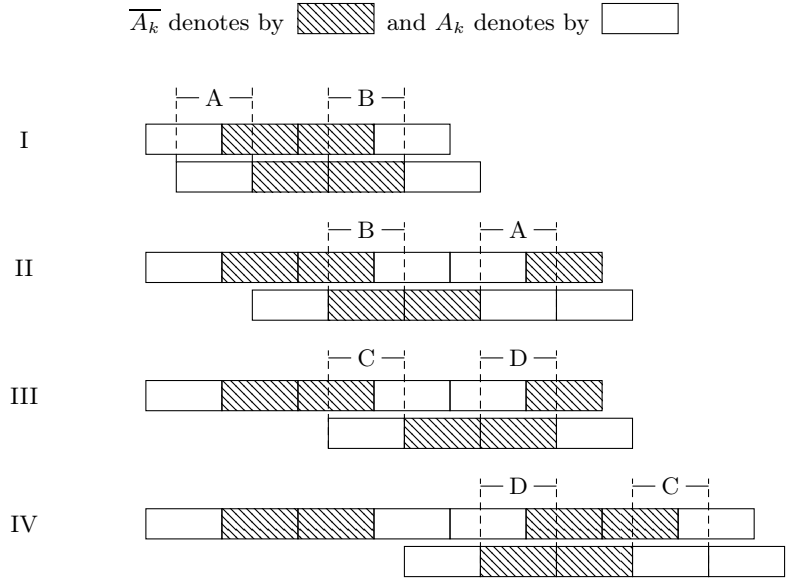
\begin{figure}[h]\label{figure2}
\centering
\begin{tikzpicture}[scale=5]
\draw[pattern=north west lines] (0,.62) rectangle (.2,.7);
\node [label={[xshift=-1.2cm, yshift=2.8cm]\scriptsize $\overline{A_k}$ denotes by}] {};
\draw(.8,.62) rectangle (1,.7);
\node [label={[xshift=2.5cm, yshift=2.8cm]\scriptsize and $A_k$ denotes by}] {};
\draw (-.4,.3) rectangle (-.2,.38);
\draw [pattern=north west lines] (-.2,.3) rectangle (0,.38);
\draw [pattern=north west lines] (0,.3) rectangle (.2,.38);
\draw (.2,.3) rectangle (.4,.38);

\draw (-.32,.2) rectangle (-.12,.28);
\draw [pattern=north west lines] (-.12,.2) rectangle (.08,.28);
\draw [pattern=north west lines](.08,.2) rectangle (.28,.28);
\draw (.28,.2) rectangle (.48,.28);

\draw [densely dashed](-.32,.28)--(-.32,.48);
\draw [densely dashed](-.12,.28)--(-.12,.48);
\draw [densely dashed](.08,.28)--(.08,.48);
\draw [densely dashed](.28,.28)--(.28,.48);
\draw (-.32,.45)--(-.26,.45);
\draw (-.18,.45)--(-.12,.45);
\draw(-.22,.45)node{\scriptsize A};
\draw (.08,.45)--(.14,.45);
\draw (.22,.45)--(.28,.45);
\draw(.18,.45)node{\scriptsize B};

\draw (-.4,-.04) rectangle (-.2,.04);
\draw [pattern=north west lines] (-.2,-.04) rectangle (0,.04);
\draw [pattern=north west lines] (0,-.04) rectangle (.2,.04);
\draw (.2,-.04) rectangle (.4,.04);
\draw (.4,-.04) rectangle (.6,.04);
\draw [pattern=north west lines] (.6,-.04) rectangle (.8,.04);

\draw (-.12,-.06) rectangle (.08,-.14);
\draw [pattern=north west lines] (.08,-.06) rectangle (.28,-.14);
\draw [pattern=north west lines] (.28,-.06) rectangle (.48,-.14);
\draw (.48,-.06) rectangle (.68,-.14);
\draw (.68,-.06) rectangle (.88,-.14);

\draw [densely dashed](.08,-.14)--(.08,.12);
\draw [densely dashed](.28,-.14)--(.28,.12);
\draw [densely dashed](.48,-.14)--(.48,.12);
\draw [densely dashed](.68,-.14)--(.68,.12);

\draw (.08,.1)--(.14,.1);
\draw (.22,.1)--(.28,.1);
\draw(.18,.1)node{\scriptsize B};
\draw (.48,.1)--(.54,.1);
\draw (.62,.1)--(.68,.1);
\draw(.58,.1)node{\scriptsize A};

\draw (-.4,-.38) rectangle (-.2,-.3);
\draw [pattern=north west lines] (-.2,-.38) rectangle (0,-.3);
\draw [pattern=north west lines] (0,-.38) rectangle (.2,-.3);
\draw (.2,-.38) rectangle (.4,-.3);
\draw (.4,-.38) rectangle (.6,-.3);
\draw [pattern=north west lines] (.6,-.38) rectangle (.8,-.3);

\draw (.08,-.48) rectangle (.28,-.4);
\draw [pattern=north west lines] (.28,-.48) rectangle (.48,-.4);
\draw [pattern=north west lines] (.48,-.48) rectangle (.68,-.4);
\draw (.68,-.48) rectangle (.88,-.4);

\draw [densely dashed](.08,-.48)--(.08,-.22);
\draw [densely dashed](.28,-.48)--(.28,-.22);
\draw [densely dashed](.48,-.48)--(.48,-.22);
\draw [densely dashed](.68,-.48)--(.68,-.22);

\draw (.08,-.24)--(.14,-.24);
\draw (.22,-.24)--(.28,-.24);
\draw(.18,-.24)node{\scriptsize C};
\draw (.48,-.24)--(.54,-.24);
\draw (.62,-.24)--(.68,-.24);
\draw(.58,-.24)node{\scriptsize D};

\draw (-.4,-.72) rectangle (-.2,-.64);
\draw [pattern=north west lines] (-.2,-.72) rectangle (0,-.64);
\draw [pattern=north west lines] (0,-.72) rectangle (.2,-.64);
\draw (.2,-.72) rectangle (.4,-.64);
\draw (.4,-.72) rectangle (.6,-.64);
\draw [pattern=north west lines] (.6,-.72) rectangle (.8,-.64);
\draw [pattern=north west lines] (.8,-.72) rectangle (1,-.64);
\draw (1,-.72) rectangle (1.2,-.64);

\draw (.28,-.82) rectangle (.48,-.74);
\draw [pattern=north west lines] (.48,-.82) rectangle (.68,-.74);
\draw [pattern=north west lines] (.68,-.82) rectangle (.88,-.74);
\draw (.88,-.82) rectangle (1.08,-.74);
\draw (1.08,-.82) rectangle (1.28,-.74);

\draw [densely dashed](.48,-.82)--(.48,-.56);
\draw [densely dashed](.68,-.82)--(.68,-.56);
\draw [densely dashed](.88,-.82)--(.88,-.56);
\draw [densely dashed](1.08,-.82)--(1.08,-.56);

\draw (.48,-.58)--(.54,-.58);
\draw (.62,-.58)--(.68,-.58);
\draw(.58,-.58)node{\scriptsize D};
\draw (.88,-.58)--(.94,-.58);
\draw (1.02,-.58)--(1.08,-.58);
\draw(.98,-.58)node{\scriptsize C};

\node [label={[xshift=-3.6cm, yshift=1.3cm]\scriptsize I}] {};
\node [label={[xshift=-3.6cm, yshift=-.4cm]\scriptsize II}] {};
\node [label={[xshift=-3.6cm, yshift=-2.1cm]\scriptsize III}] {};
\node [label={[xshift=-3.6cm, yshift=-3.8cm]\scriptsize IV}] {};
\end{tikzpicture}\caption{ I: $0<i_2<2^r$; II: $2^r<i_2<2^{r+1}$; III: $2^{r+1}<i_2<2^{r+1}+2^r$ and IV: $2^{r+1}+2^r<i_2<2^{r+2}$.}
\end{figure}
\end{proof}

 \begin{lemma}\label{20}
 Let $n\geq 3$. Then for each $0<i<2^{n+1}$, there exists $0<u<2^{n+2}, u\ne 2^{n+1}$ such that the $u$-th term of
 $$
 (\sigma^i((\eps_n\overline{\eps_n})^\infty),(\eps_n\overline{\eps_n}^+\overline{\eps_n}\eps_n^-)^\infty)
 $$
belongs to $\{(1,1), (-1,-1)\}$.
 \end{lemma}
\begin{proof}
For simplicity, denote
$$
\Gamma _i:=(\sigma^i((\eps_n\overline{\eps_n})^\infty),(\eps_n\overline{\eps_n}^+
\overline{\eps_n}\eps_n^-)^\infty).
$$
For a sequence $\alpha =a_1a_2\cdots $ and positive integer $m\le n$, let $\alpha (m,n)=a_m\cdots a_n$.
We divide the proof into three cases.

Case I. $0<i<2^n$. In this case either the segment $\Gamma _i(2^n-i+1, 2\cdot 2^{n}-i)$ or
the segment $\Gamma _i(3\cdot2^{n}-i+1, 4\cdot2^{n}-i)$ contains $(1,1)$ or $(-1,-1)$.

In fact, both these two blocks have the same upper part but the lower part of the former segment is the reflection of  the lower part of the latter segment. More exactly, if we denote
 $$
 \Gamma _i(2^n-i+1, 2\cdot 2^{n}-i)=(c_1\cdots c_{2^n}, d_1\cdots d_{2^n}),
 $$
 then
 $$
 \Gamma _i(3\cdot2^{n}-i+1, 4\cdot2^{n}-i)=(c_1\cdots c_{2^n}, \overline{d_1\cdots d_{2^n}}).
 $$
 On the other hand, when $i$ is even, each odd term $(c_{2s-1}, d_{2s-1})$ satisfies $c_{2s-1}d_{2s-1}\ne 0$.
 When $i$ is odd, there exist $2s$ such that $c_{2s}d_{2s}\ne 0$ by (P4) (see Figure 3).

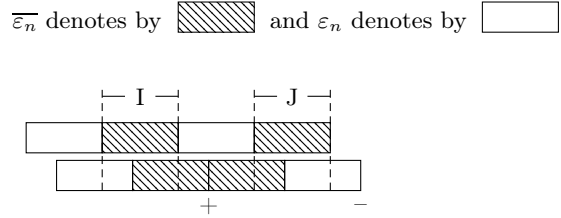
\begin{figure}[h]\label{figure3}
\centering
\begin{tikzpicture}[scale=5]
\draw[pattern=north west lines] (0,.62) rectangle (.2,.7);
\node [label={[xshift=-1.2cm, yshift=2.8cm]\scriptsize $\overline{\eps_n}$ denotes by}] {};
\draw(.8,.62) rectangle (1,.7);
\node [label={[xshift=2.5cm, yshift=2.8cm]\scriptsize and $\eps_n$ denotes by}] {};
\draw (-.4,.3) rectangle (-.2,.38);
\draw [pattern=north west lines] (-.2,.3) rectangle (0,.38);
\draw (0,.3) rectangle (.2,.38);
\draw [pattern=north west lines] (.2,.3) rectangle (.4,.38);

\draw (-.32,.2) rectangle (-.12,.28);
\draw [pattern=north west lines] (-.12,.2) rectangle (.08,.28);
\draw [pattern=north west lines] (.08,.2) rectangle (.28,.28);
\draw (.28,.2) rectangle (.48,.28);

\draw [densely dashed](-.2,.2)--(-.2,.48);
\draw [densely dashed](0,.2)--(0,.48);
\draw [densely dashed](.2,.2)--(.2,.48);
\draw [densely dashed](.4,.2)--(.4,.48);

\draw (-.2,.45)--(-.14,.45);
\draw (-.06,.45)--(0,.45);
\draw(-.1,.45)node{\scriptsize I};
\draw (.2,.45)--(.26,.45);
\draw (.34,.45)--(.4,.45);
\draw(.3,.45)node{\scriptsize J};
\draw(.08,.16)node{\scriptsize +};
\draw(.48,.16)node{\scriptsize --};




\end{tikzpicture}\caption{Case I: $0<i<2^n$.}
\end{figure}

 Case II. $2^n<i<2^{n+1}$. In this case either the segment $\Gamma _i(2\cdot2^{n}-i+1, 3\cdot 2^{n}-i)$ or
the segment $\Gamma _i(4\cdot2^{n}-i+1, 5\cdot2^{n}-i)$ contains $(1,1)$ or $(-1,-1)$.
This can be reduced by the same way as in Case I (see Figure 4).

\begin{figure}[h]\label{figure4}
\centering
\begin{tikzpicture}[scale=5]
\draw[pattern=north west lines] (0,.62) rectangle (.2,.7);
\node [label={[xshift=-1.2cm, yshift=2.8cm]\scriptsize $\overline{\eps_n}$ denotes by}] {};
\draw(.8,.62) rectangle (1,.7);
\node [label={[xshift=2.5cm, yshift=2.8cm]\scriptsize and $\eps_n$ denotes by}] {};
\draw (-.4,.3) rectangle (-.2,.38);
\draw [pattern=north west lines] (-.2,.3) rectangle (0,.38);
\draw (0,.3) rectangle (.2,.38);
\draw [pattern=north west lines] (.2,.3) rectangle (.4,.38);
\draw (.4,.3) rectangle (.6,.38);

\draw (-.12,.2) rectangle (.08,.28);
\draw [pattern=north west lines] (.08,.2) rectangle (.28,.28);
\draw [pattern=north west lines] (.28,.2) rectangle (.48,.28);
\draw (.48,.2) rectangle (.68,.28);

\draw [densely dashed](0,.2)--(0,.48);
\draw [densely dashed](.2,.2)--(.2,.48);
\draw [densely dashed](.4,.2)--(.4,.48);
\draw [densely dashed](.6,.2)--(.6,.48);

\draw (0,.45)--(.06,.45);
\draw (.14,.45)--(.2,.45);
\draw(.1,.45)node{\scriptsize I};
\draw (.4,.45)--(.46,.45);
\draw (.54,.45)--(.6,.45);
\draw(.5,.45)node{\scriptsize J};
\draw(.28,.16)node{\scriptsize +};
\draw(.68,.16)node{\scriptsize --};

\end{tikzpicture}\caption{Case II: $2^n<i<2^{n+1}$.}
\end{figure}
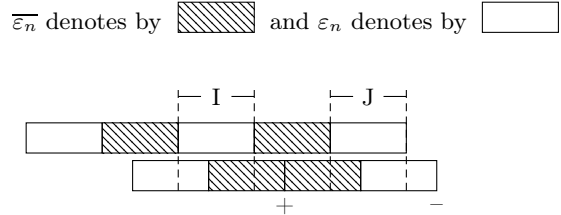
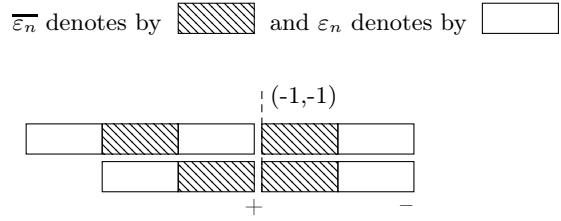
\begin{figure}[h]\label{figure5}
\centering
\begin{tikzpicture}[scale=5]
\draw[pattern=north west lines] (0,.62) rectangle (.2,.7);
\node [label={[xshift=-1.2cm, yshift=2.8cm]\scriptsize $\overline{\eps_n}$ denotes by}] {};
\draw(.8,.62) rectangle (1,.7);
\node [label={[xshift=2.5cm, yshift=2.8cm]\scriptsize and $\eps_n$ denotes by}] {};
\draw (-.4,.3) rectangle (-.2,.38);
\draw [pattern=north west lines] (-.2,.3) rectangle (0,.38);
\draw (0,.3) rectangle (.2,.38);
\draw [pattern=north west lines] (.22,.3) rectangle (.42,.38);
\draw (.42,.3) rectangle (.62,.38);

\draw (-.2,.2) rectangle (0,.28);
\draw [pattern=north west lines] (0,.2) rectangle (.2,.28);
\draw [pattern=north west lines] (.22,.2) rectangle (.42,.28);
\draw (.42,.2) rectangle (.62,.28);

\draw [densely dashed](.22,.2)--(.22,.48);

\draw(.33,.45)node{\scriptsize (-1,-1)};
\draw(.2,.16)node{\scriptsize +};
\draw(.6,.16)node{\scriptsize --};
\end{tikzpicture}\caption{Case III: $i=2^n$.}
\end{figure}

Case III. $i=2^n$. In this case we have
$$
 \Gamma _{2^n}(2^{n+1}+1,  2^{n+1}+2^n)=(\overline{\eps_n}, \overline{\eps_n}),
 $$
 whose first term is $(-1,-1)$ (see Figure 5).
\end{proof}
\begin{remark}
The conclusion of Lemma \ref{20} is still correct if $\eps_n\overline{\eps_n}^+\overline{\eps_n}\eps_n^-$ is
replaced by $\eps_n\overline{\eps_n}^+\overline{\eps_n}\eps_n$.
\end{remark}

Recall that
\begin{equation*}
E_{n,m}^i=(\sigma^i((\eps_n\overline{\eps_n})^\infty),(\eps_m\overline{\eps_m})^\infty).
\end{equation*}
We have the following

\begin{lemma}\label{15}
Let $m\ge n\ge1$. There exists some $0<i<2^{n+1}$ such that the sequence $E_{n,m}^i\in\Omega_2^{\mathbb N}$
 if and only if $m=n$.
\end{lemma}
\begin{proof}
Let $m=n$.  The sufficiency follows from Lemma \ref{12}(i).

Now suppose that $n<m$. It suffices to show that the sequence $E_{n,m}^i$  contains either $(-1,-1)$ or $(1,1)$ for all $0<i<2^{n+1}$. This can be directly verified for $n=1,2$. Let $n\geq 3$.  Note that when $n<m$, the
sequence $(\eps_m\overline{\eps_m})^\infty$ begins with
\begin{equation*}
\eps_n\overline{\eps_n}^+\overline{\eps_n}\eps_n^-\qtq{or}\eps_n\overline{\eps_n}^+\overline{\eps_n}\eps_n.
\end{equation*}
From Lemma \ref{20} and the consequent remark it follows that the sequence $E_{n,m}^i$  contains either $(-1,-1)$ or $(1,1)$ for all $0<i<2^{n+1}$.
\end{proof}

Applying Lemmas \ref{9}, \ref{12} and \ref{15} we have the following results.
\begin{lemma}\label{16}
Let $q\in(q_m,q_{m+1}]$, $m\ge1$ and $t\in T$ with $q$-expansion $(t_i)=((t_{i,1}), (t_{i,2}))$.
Suppose that  $(t_{i,1}), (t_{i,2})\notin \set{(-1)^\infty, 1^\infty}$. If $(t_i)$ contains infinitely many $(0,0)$, then
  either one of $(t_{i,1})$ and $(t_{i,2})$ ends with $0^\infty$  or $(t_i)$ ends with
 $((\eps_n\overline{\eps_n})^\infty,(\overline{\eps_n}\eps_n)^\infty)$ for some $1\le n<m$.
\end{lemma}
\begin{remark}
We omit the cases $(t_{i,1})$ or $(t_{i,2})$ can end with $(\eps_0\overline{\eps_0})^\infty=(1(-1))^\infty$ under the assumption of Lemmas \ref{16} and \ref{l35}, otherwise $(t_i)$ will contain at most finitely many $(0,0)$.
\end{remark}
\begin{proof}
Suppose that  $(t_{i,1}), (t_{i,2})\notin \set{(-1)^\infty, 1^\infty}$. It is important that  $(t_i)$ contains neither $(-1,-1)$ nor $(1,1)$.

From Lemma \ref{9} it follows that $(t_{i,1}),(t_{i,2})$
end with one of
\begin{equation*}
0^\infty, (\eps_1\overline{\eps_1})^\infty,(\eps_2\overline{\eps_2})^\infty,\cdots,(\eps_{m-1}\overline{\eps_{m-1}})^\infty.
\end{equation*}
Assume that both $(t_{i,1})$ and $(t_{i,2})$ don't end with $0^\infty $. Let $(t_{i,1})$ ends with
$(\eps_\alpha \overline{\eps_\alpha })^\infty$, and  $(t_{i,2})$ ends with
$(\eps_\beta \overline{\eps_\beta })^\infty$. Lemma \ref{15} implies that $t = (t_i)$ ends with $(\sigma^i(\eps_{\alpha}\overline{\eps_{\alpha}})^\infty, (\eps_{\alpha}\overline{\eps_{\alpha}})^\infty)$. Since $(t_i)$
contains infinitely many $(0,0)$, then Lemma \ref{12} implies that $i=2^\alpha$, so
\begin{equation*}
(\sigma^i(\eps_{\alpha}\overline{\eps_{\alpha}})^\infty,(\eps_{\alpha}\overline{\eps_{\alpha}})^\infty)=((\overline{\eps_{\alpha}}\eps_{\alpha})^\infty,(\eps_{\alpha}\overline{\eps_{\alpha}})^\infty).
\end{equation*}
Note that ending with a tail of $((\overline{\eps_{\alpha}}\eps_{\alpha})^\infty,(\eps_{\alpha}\overline{\eps_{\alpha}})^\infty)$
is equivalent to ending with a tail of $((\eps_{\alpha}\overline{\eps_{\alpha}})^\infty,(\overline{\eps_{\alpha}}\eps_{\alpha})^\infty)$.
\end{proof}
\begin{lemma}\label{l35}
Let $q=q_{KL}$, $t\in T$ with $q$-expansion $(t_i)=((t_{i,1}), (t_{i,2}))$.
Suppose that  $(t_{i,1}), (t_{i,2})\notin \set{(-1)^\infty, 1^\infty}$. If $(t_i)$ contains infinitely many $(0,0)$, then
  either one of $(t_{i,1})$ and $(t_{i,2})$ ends with $0^\infty$  or $(t_i)$ ends with
 $((\eps_n\overline{\eps_n})^\infty,(\overline{\eps_n}\eps_n)^\infty)$ for $n\ge1$.
\end{lemma}
\begin{proof}
Let $q=q_{KL}$ and $t\in T$ with $q$-expansion $(t_i)=((t_{i,1}), (t_{i,2}))$.
Suppose that  $(t_{i,1}), (t_{i,2})\notin \set{(-1)^\infty, 1^\infty}$ and $(t_i)$ contains infinitely many $(0,0)$. We distinguish two cases.

Case I. $(t_{i,1})$ or $(t_{i,2})$ ends with $0^\infty$. We assume that $(t_{i,1})$ ends with $0^\infty$. Then by Lemma \ref{10} $(t_{i,2})$ ends with one of
\begin{equation*}
0^\infty, (\eps_1\overline{\eps_1})^\infty,(\eps_2\overline{\eps_2})^\infty,\cdots,
\end{equation*}
or ends with the sequence of the form
\begin{equation}\label{e26}
(\eps_{0}\overline{\eps_{0}})^{j_0}(\eps_{0}\overline{\eps_{1}})^{l_0}(\eps_{1}\overline{\eps_{1}})^{j_1}(\eps_{1}\overline{\eps_{2}})^{l_1}\cdots
\end{equation}
and its reflection, where
\begin{equation*}
l_m\in\set{0,1}\quad\text{and}\quad 0\le j_m<\infty \quad\text{for all $m\ge0$}.
\end{equation*}

Case II. $(t_{i,1})$ and $(t_{i,2})$ do not end with $0^\infty$. By Lemmas \ref{10}, \ref{12} and \ref{15} it remains to consider the cases that $(t_{i,1})$ ends with one of
\begin{equation*}
(\eps_1\overline{\eps_1})^\infty,(\eps_2\overline{\eps_2})^\infty,\cdots,
\end{equation*}
and $(t_{i,2})$ ends with the sequence having the form of \eqref{e26}, and both $(t_{i,1}),(t_{i,2})$ end with the sequence having the form of \eqref{e26}.

II (i). $(t_{i,1})$ ends with $(\eps_n\overline{\eps_n})^\infty$ for some $n\ge1$ and $(t_{i,2})$ ends with the sequence of the form of \eqref{e26}.

Note that $\eps_{n+2}=\eps_n\overline{\eps_n}^+\overline{\eps_n}\eps_n$ appears infinitely often in $(t_{i,2})$, because $\eps_{m}$ begins with $\eps_{n+2}$ for all $m\ge n+2$.
It follows from Lemma \ref{20} and the consequent remark that $(t_{i,1})$ and $(t_{i,2})$ are not $\Omega_2$-matched, which leads to contradiction.

II (ii). Both $(t_{i,1})$ and $(t_{i,2})$ end with the sequence of the form of \eqref{e26}. We choose $k_1,k_2$ large enough such that $(t_{i+k_1,1})$ and $(t_{i+k_2,2})$ begin with one of the following blocks for some large positive integer $n$
\begin{equation}\label{e32}
B_1=\eps_n\overline{\eps_n}^+\overline{\eps_n}\eps_n^-, B_2=\eps_n\overline{\eps_n}^+\overline{\eps_n}\eps_n, B_3=\overline{\eps_n}\eps_n^-\eps_n\overline{\eps_n}^+, B_4=\overline{\eps_n}\eps_n^-\eps_n\overline{\eps_n}.
\end{equation}
Then next block of length $2^{n+2}$ is one of \eqref{e32}, i.e., $(t_{i+k_1,1})$ and $(t_{i+k_2,2})$ start at
\begin{equation}\label{e28}
B_{i_1}B_{i_2}B_{i_3}\cdots
\end{equation}
where $i_n\in\set{1,2,3,4}$ for all $n\ge1$. We take the same method as in the proof of Lemma \ref{12}, removing the even terms in $B_{1}, B_{2}, B_{3}, B_{4}$, and denote by $C_1=A_{n-1}\overline{A_{n-1}}\overline{A_{n-1}}A_{n-1}, C_2=\overline{A_{n-1}}A_{n-1}A_{n-1}\overline{A_{n-1}}$ the new blocks. It follows from \eqref{e28} that there are $C_1C_1, C_1C_2, C_2C_1, C_2C_2$ four cases. Next we consider the case $(\sigma^i(C_1C_2),C_1C_2)$ for $0<i<2^{n+1}$ (see Figures 1 and 6, let $k=n-1$ in Figure 1), for the cases $i=2^{n-1},2^n,2^{n-1}+2^n$,
       one can check directly by Figure 6. The remaining fifteen cases are similar. Hence $(t_{i,1})$ and $(t_{i,2})$ are not $\Omega_2$-matched, which leads to the same contradiction.
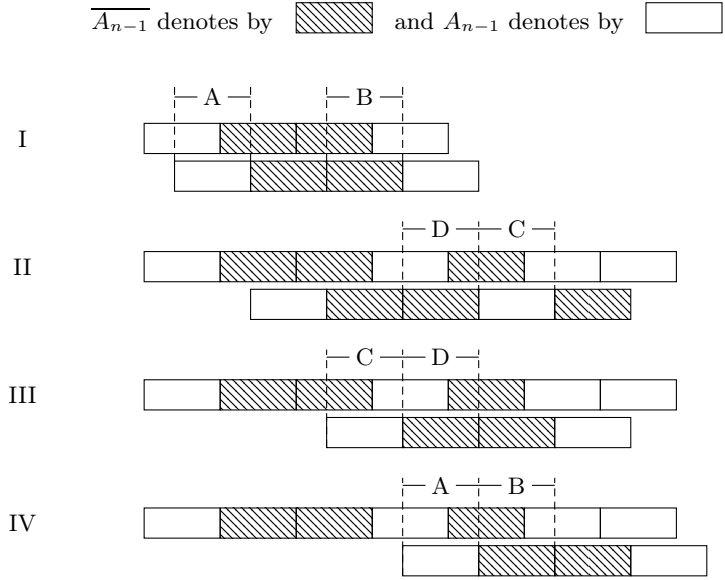
\begin{figure}[h]\label{figure6}
\centering
\begin{tikzpicture}[scale=5]
\draw[pattern=north west lines] (0,.62) rectangle (.2,.7);
\node [label={[xshift=-1.5cm, yshift=2.8cm]\scriptsize $\overline{A_{n-1}}$ denotes by}] {};
\draw(.92,.62) rectangle (1.12,.7);
\node [label={[xshift=2.8cm, yshift=2.8cm]\scriptsize and $A_{n-1}$ denotes by}] {};
\draw (-.4,.3) rectangle (-.2,.38);
\draw [pattern=north west lines] (-.2,.3) rectangle (0,.38);
\draw [pattern=north west lines] (0,.3) rectangle (.2,.38);
\draw (.2,.3) rectangle (.4,.38);

\draw (-.32,.2) rectangle (-.12,.28);
\draw [pattern=north west lines] (-.12,.2) rectangle (.08,.28);
\draw [pattern=north west lines](.08,.2) rectangle (.28,.28);
\draw (.28,.2) rectangle (.48,.28);

\draw [densely dashed](-.32,.28)--(-.32,.48);
\draw [densely dashed](-.12,.28)--(-.12,.48);
\draw [densely dashed](.08,.28)--(.08,.48);
\draw [densely dashed](.28,.28)--(.28,.48);
\draw (-.32,.45)--(-.26,.45);
\draw (-.18,.45)--(-.12,.45);
\draw(-.22,.45)node{\scriptsize A};
\draw (.08,.45)--(.14,.45);
\draw (.22,.45)--(.28,.45);
\draw(.18,.45)node{\scriptsize B};

\draw (-.4,-.04) rectangle (-.2,.04);
\draw [pattern=north west lines] (-.2,-.04) rectangle (0,.04);
\draw [pattern=north west lines] (0,-.04) rectangle (.2,.04);
\draw (.2,-.04) rectangle (.4,.04);
\draw [pattern=north west lines](.4,-.04) rectangle (.6,.04);
\draw  (.6,-.04) rectangle (.8,.04);
\draw  (.8,-.04) rectangle (1,.04);

\draw (-.12,-.06) rectangle (.08,-.14);
\draw [pattern=north west lines] (.08,-.06) rectangle (.28,-.14);
\draw [pattern=north west lines] (.28,-.06) rectangle (.48,-.14);
\draw (.48,-.06) rectangle (.68,-.14);
\draw [pattern=north west lines](.68,-.06) rectangle (.88,-.14);

\draw [densely dashed](.28,-.14)--(.28,.12);
\draw [densely dashed](.48,-.14)--(.48,.12);
\draw [densely dashed](.68,-.14)--(.68,.12);

\draw (.28,.1)--(.34,.1);
\draw (.48,.1)--(.42,.1);
\draw(.38,.1)node{\scriptsize D};
\draw (.48,.1)--(.54,.1);
\draw (.62,.1)--(.68,.1);
\draw(.58,.1)node{\scriptsize C};

\draw (-.4,-.38) rectangle (-.2,-.3);
\draw [pattern=north west lines] (-.2,-.38) rectangle (0,-.3);
\draw [pattern=north west lines] (0,-.38) rectangle (.2,-.3);
\draw (.2,-.38) rectangle (.4,-.3);
\draw [pattern=north west lines](.4,-.38) rectangle (.6,-.3);
\draw (.6,-.38) rectangle (.8,-.3);
\draw (.8,-.38) rectangle (1,-.3);

\draw (.08,-.48) rectangle (.28,-.4);
\draw [pattern=north west lines] (.28,-.48) rectangle (.48,-.4);
\draw [pattern=north west lines] (.48,-.48) rectangle (.68,-.4);
\draw (.68,-.48) rectangle (.88,-.4);

\draw [densely dashed](.08,-.48)--(.08,-.22);
\draw [densely dashed](.28,-.48)--(.28,-.22);
\draw [densely dashed](.48,-.48)--(.48,-.22);

\draw (.08,-.24)--(.14,-.24);
\draw (.22,-.24)--(.28,-.24);
\draw(.18,-.24)node{\scriptsize C};
\draw (.28,-.24)--(.34,-.24);
\draw (.42,-.24)--(.48,-.24);
\draw(.38,-.24)node{\scriptsize D};

\draw (-.4,-.72) rectangle (-.2,-.64);
\draw [pattern=north west lines] (-.2,-.72) rectangle (0,-.64);
\draw [pattern=north west lines] (0,-.72) rectangle (.2,-.64);
\draw (.2,-.72) rectangle (.4,-.64);
\draw [pattern=north west lines](.4,-.72) rectangle (.6,-.64);
\draw  (.6,-.72) rectangle (.8,-.64);
\draw  (.8,-.72) rectangle (1,-.64);

\draw (.28,-.82) rectangle (.48,-.74);
\draw [pattern=north west lines] (.48,-.82) rectangle (.68,-.74);
\draw [pattern=north west lines] (.68,-.82) rectangle (.88,-.74);
\draw (.88,-.82) rectangle (1.08,-.74);

\draw [densely dashed](.48,-.82)--(.48,-.56);
\draw [densely dashed](.68,-.82)--(.68,-.56);
\draw [densely dashed](.28,-.82)--(.28,-.56);

\draw (.48,-.58)--(.54,-.58);
\draw (.62,-.58)--(.68,-.58);
\draw(.58,-.58)node{\scriptsize B};
\draw (.28,-.58)--(.34,-.58);
\draw (.42,-.58)--(.48,-.58);
\draw(.38,-.58)node{\scriptsize A};

\node [label={[xshift=-3.6cm, yshift=1.3cm]\scriptsize I}] {};
\node [label={[xshift=-3.6cm, yshift=-.4cm]\scriptsize II}] {};
\node [label={[xshift=-3.6cm, yshift=-2.1cm]\scriptsize III}] {};
\node [label={[xshift=-3.6cm, yshift=-3.8cm]\scriptsize IV}] {};
\end{tikzpicture}\caption{ I: $0<i\le2^{n-1}$; II: $2^{n-1}<i\le2^{n}$; III: $2^{n}<i\le2^{n-1}+2^n$ and IV: $2^{n-1}+2^n<i<2^{n+1}$.}
\end{figure}
\end{proof}

The following result was proved in \cite[p. 2829]{DLD} (see also \cite[Lemma 3.5]{BD}).
\begin{lemma}\label{5}
Let $q\in(q_{KL},3)$, then there exists $n\in\mathbb N$ such that $\uu_q'$ contains the subshift of
fnite type over the alphabet $\mathcal A=\set{a_n,b_n,\overline{a_n},\overline{b_n}}$ with transition matrix
\begin{equation*}
  A=\left (
  \begin{array}{llll}
  0 & 1 & 1 & 0   \\
  0 & 0 & 1 & 0  \\
1 & 0 & 0 & 1  \\
  1 & 0 & 0& 0  \\
  \end{array}
  \right ),
\end{equation*}
where $a_n=0\lambda_1\cdots \lambda_{2^n-1}$ and $b_n=(-1)\lambda_1\cdots \lambda_{2^n-1}$.
\end{lemma}
For reader's convenience, we restate our result.
\begin{2}\mbox{}
\begin{enumerate}[\upshape (i)]
\item If $q\in(q_m,q_{m+1}]$ for some $m\ge1$, then
\begin{equation*}
D_q=\set{0,\frac{\log3}{\log q}}\bigcup\set{-\frac{\log3}{\log q}\sum_{i=1}^{n}\left(-\frac{1}{2}\right)^i:1\le n<m}.
\end{equation*}
\item If $q=q_{KL}$ then
\begin{align*}
 D_{q_{KL}}&=\set{0,\frac{\log3}{\log q_{KL}},\frac{\log3}{3\log q_{KL}}}\\
 &\bigcup\set{-\frac{\log3}{\log q_{KL}}\sum_{i=1}^{n}
 \left(-\frac{1}{2}\right)^i:1\le n<\infty}.
\end{align*}

\item If $q\in(q_{KL},3)$, $D_q$ contains an interval.
\end{enumerate}
\end{2}
\begin{proof}[Proof of Theorem \ref{2}]
Let $t\in T$ with $q$-expansion $(t_i)=((t_{i,1}), (t_{i,2}))$. If $(t_{i,1})$ or $(t_{i,2})$ is in $\set{(-1)^\infty,1^\infty}$, then one
gets that $0\in D_{\alpha}$, so we assume that $(t_{i,1})$ and $(t_{i,2})$ are not in $\set{(-1)^\infty,1^\infty}$ and $(t_i)$ contains infinitely many $(0,0)$.

(i) Let $q\in(q_m,q_{m+1}]$ for some $m\ge1$.
 By Lemma \ref{9}, $(t_{i,1})$ and $(t_{i,2})$ is eventually periodic. It follows from Lemma \ref{15}, $(t_i)=((t_{i,1}), (t_{i,2}))$ is a sequence of form
\begin{equation*}
(\omega_1(\eps_n\overline{\eps_n})^\infty,\omega_2(\eps_n\overline{\eps_n})^\infty)
\end{equation*}
or
\begin{equation*}
(\omega_1(\eps_n\overline{\eps_n})^\infty,\omega_20^\infty)\qtq{or}(\omega_10^\infty,\omega_20^\infty)
\end{equation*}
for some $1\le n<m$. As mentioned in the preceding remark we exclude the situation $n=0$. Then from Lemma \ref{16} we infer that $(t_i)$ ends with
\begin{equation*}
((\eps_n\overline{\eps_n})^\infty,(\overline{\eps_n}\eps_n)^\infty),
\end{equation*}
or
\begin{equation*}
((\eps_n\overline{\eps_n})^\infty,0^\infty)\qtq{or}(0^\infty,0^\infty)
\end{equation*}
for some $1\le n<m$.  Therefore applying Lemmas \ref{17} and \ref{7} we conclude that
\begin{align*}
&\dim _H(E\cap (E+t))\\
=&
\begin{cases}
-\frac{(\log3)[\sum_{i=1}^n(-1/2)^i]}{\log q}&\text{if $(t_i)$ ends with $((\eps_n\overline{\eps_n})^\infty,(\overline{\eps_n}\eps_n)^\infty)$,}\\
-\frac{(\log3)[\sum_{i=1}^n(-1/2)^i]}{\log q}&\text{if $(t_i)$ ends with $((\eps_n\overline{\eps_n})^\infty,0^\infty)$,}\\
\quad\log3/\log q&\text{if $(t_i)$ ends with $(0^\infty,0^\infty)$.}
\end{cases}
\end{align*}

(ii) Thanks to Lemma \ref{l35} and (i), it suffices to prove the case $(t_{i,1})$ ends with $0^\infty$ and $(t_{i,2})$ ends with the sequence of the form
\begin{equation}\label{e33}
(\eps_{0}\overline{\eps_{0}})^{j_0}(\eps_{0}\overline{\eps_{1}})^{l_0}(\eps_{1}\overline{\eps_{1}})^{j_1}(\eps_{1}\overline{\eps_{2}})^{l_1}\cdots
\end{equation}
 where
\begin{equation*}
l_m\in\set{0,1}\quad\text{and}\quad 0\le j_m<\infty \quad\text{for all $m\ge0$}.
\end{equation*}
We claim that the density of $0$ in the sequence of the form \eqref{e33} is $1/3$. We arbitrarily take a sequence $(t_{i,2})$ of the form \eqref{e33}. Given a  positive integer $n$ arbitrarily, by the structure of $\eps_n $ (see (P1))
 one can get there exists $k=k(n)$ such that
$$
\sigma ^k(t_{i,2})=B_{i_1}B_{i_2}B_{i_3}\cdots
$$
where $i_\ell \in \{1,2,3,4\}$ and $B_{i_\ell}$ is given by \eqref{e32}.

%
%

On the other hand, an application of (P2) and Lemma \ref{7} yields that when $n$ is odd
\begin{align*}
&d^*(B_1^\infty)=d^*(B_3^\infty)=-\sum_{i=1}^n\left (-\frac{1}{2}\right )^i-\frac{1}{2^{n+1}}=\frac{1}{3}-\frac{1}{3\cdot 2^{n+1}},\\
&d^*(B_2^\infty)=d^*(B_4^\infty)=-\sum_{i=1}^n\left (-\frac{1}{2}\right )^i-\frac{1}{2^{n+2}}=\frac{1}{3}+\frac{1}{3\cdot 2^{n+2}},
\end{align*}
and when $n$ is even
\begin{align*}
&d^*(B_1^\infty)=d^*(B_3^\infty)=-\sum_{i=1}^n\left (-\frac{1}{2}\right )^i+\frac{1}{2^{n+1}}=\frac{1}{3}+\frac{1}{3\cdot 2^{n+1}},\\
&d^*(B_2^\infty)=d^*(B_4^\infty)=-\sum_{i=1}^n\left (-\frac{1}{2}\right )^i+\frac{1}{2^{n+2}}=\frac{1}{3}-\frac{1}{3\cdot 2^{n+2}}.
\end{align*}
Note that $d^*(t_{i,2})=d^*(\sigma ^k(t_{i,2}))$. By above considerations one have
$$
\left |d^*(t_{i,2})-\frac{1}{3}\right |=\left |d^*(\sigma ^k(t_{i,2}))-\frac{1}{3}\right |<\frac{1}{3\cdot 2^{n+1}}.
$$
Since $n$ can be taken arbitrarily, from which the claim follows. So for this case we have $\dim _H(E\cap (E+t))=\frac{\log 3}{3\log q_{KL}}$ by Lemma \ref{17}.

(iii) Let $q\in(q_{KL},3)$. Let $n$ be as in Lemma \ref{5}, given two words $u_1=(b_n \overline{a_n b_n} a_n,\overline{b_n} a_n b_n \overline{a_n})$ and $u_2=(\overline{a_n}a_n, a_n\overline{a_n})$, then $d(u_1)<d(u_2)$ can be verified directly. Furthermore for any $d\in[d(u_1),d(u_2)]$, there exists a sequence $(c_i)$ such that $(\eta_i)=u_1^{c_1}u_2^{c_2}\cdots$ satisfies $d((\eta_i))=d$. Hence $D_q$ contains $[\frac{\log 3}{\log q}d(u_1),\frac{\log 3}{\log q}d(u_2)]$ by Lemma \ref{17}. 
\end{proof}

\section*{Acknowledgment}
This work has been done during the first author's visit of the
Department of Mathematics of the University of Strasbourg and was
supported by the China Scholarship Council (No. 201806140142) .
He thanks the members of the department for their hospitality. Li was  supported by NSFC No.~11671147  and Science and Technology Commission of Shanghai Municipality (STCSM)  No.~18dz2271000.

 \end{document}